\documentclass[10pt, a4paper]{amsart}
\usepackage{amsfonts}
\usepackage{amscd,amsmath}
\usepackage{amssymb}
\usepackage{setspace}
\newtheorem{theorem}{Theorem}[section]
\newtheorem{lemma}[theorem]{Lemma}

\newtheorem{proposition}[theorem]{Proposition}
\newtheorem{definition}[theorem]{Definition}
\newtheorem{example}[theorem]{Example}

\theoremstyle{remark}
\newtheorem{remark}{Remark}[section]
\newtheorem*{rem*}{Remark}

\newcommand{\comment}[1]{}
\numberwithin{equation}{section}
\begin{document}

\title{Connection between commutative algebra and topology}
\author{ Sumit Kumar Upadhyay$^1$, Shiv Datt Kumar$^2$ AND RAJA SRIDHARAN$^3$\vspace{.4cm}\\
$^{1,2}$DEPARTMENT OF MATHEMATICS\\
MOTILAL NEHRU NATIONAL INSTITUTE OF TECHNOLOGY \\
ALLAHABAD, U. P., INDIA \vspace{.3cm}\\ $^3$School Of Mathematics,\\ Tata Institute of fundamental Research,\\Colaba,  Mumbai, India}
\thanks{E-mail: $^1$upadhyaysumit365@gmail.com, $^2$sdt@mnnit.ac.in, $^3$sraja@tifr.math.res.in}.\\
\thanks{ $^*$This work is supported by Council of Scientific and Industrial Research grant.}
\begin{abstract}
The main aim of this paper to show how commutative algebra is connected to topology. We give underlying topological idea of some results on completable
unimodular rows.
\end{abstract}
\maketitle
\hspace{0.8cm}\textbf{Keywords}:  Unimodular row, Matrix, Homotopy.
\section{Introduction}
Let $k$ be a field and  $\vec{x} = (x_1, x_2 \cdots, x_n)\in k^n$ be a non-zero vector, then $\vec{x}$ can be completed to a basis of $k^{n}$. We wish to have an analogue of the above statement for rings.
Let $A$ be a ring and $\vec{a} \in A^{n}, \vec{a} \ne 0$. Then there is a natural question when can $\vec{a}$ be completed to a basis of $A^{n}$?
 Suppose  $\vec{a} = (a_1, a_2, \cdots, a_n)$ can be completed to a basis  of $A^{n}$. Consider these basis vectors as columns of a matrix $\alpha$, whose first column is 
$(a_1, a_2, \cdots, a_n)$. The fact that these vectors span $A^{n}$ imply that there exists a $n\times n$ invertible matrix $\beta$ such that $\alpha\beta = I_n$.
Conversely if there exists an invertible matrix $\alpha \in M_n(A)$ with  first column $(a_1, a_2, \cdots, a_n)$, then $(a_1, a_2, \cdots, a_n)$ can be completed to a basis of $A^{n}$. 
Since any completable row $(a_1, a_2, \cdots, a_n)$ is unimodular, this leads to the following problem : 
Suppose $(a_1, a_2, \cdots, a_n) \in A^{n}$ is a unimodular row. Then  can one complete the row $(a_1, a_2, \cdots, a_n)$ to a matrix belonging to $GL_n(A)$? 
  
In general answer  of this question is negative. Surprisingly this is related to topology. 
Suppose one can find a matrix $\alpha \in GL_n(A)$ having first column $(a_1, a_2, \cdots, a_n)$. Then $e_1 \alpha^t = (a_1, a_2, \cdots, a_n)$, where $e_1 = (1, 0, \cdots, 0)$. The group $GL_n(A)$  acts on $A^{n}$ via matrix multiplication. The row $(a_1, a_2, \cdots, a_n)$ can be completed to a
matrix in $GL_n(A)$ if and only if $(a_1, a_2, \cdots, a_n)$ lies in the  orbit of $(1, 0, \cdots, 0)$ under the $GL_n(A)$ action (a similar statement holds for $SL_n(A)$).
\begin{example}\label{1}
 Let $(x_1, x_2) \in \mathbb{Z}^{2}$. Then $(x_1, x_2)$ is unimodular if and only if $x_1, x_2$ are relatively prime. In this case there exist $x, y \in \mathbb{Z}$
such that $x_{1}x + x_{2}y = 1$ and the matrix  $\left(\begin{array}{cccc} x_{1}& -y\\ x_2 & x\end{array}\right)$ has determinant $1$.
We can find an explicit completion of $(x_1, x_2)$ using the Euclidean Algorithm in the following manner: Assume for simplicity that 
$x_1, x_2 > 0$ and $x_1 > x_2$.
Then by division algorithm  $x_1 = x_{2}q + r$, where $q$ is the quotient and $r$ is the remainder. Then 
$\left(\begin{array}{cccc} 1& -q\\ 0 & 1\end{array}\right) \left(\begin{array}{cccc} x_{1}\\ x_2 \end{array}\right) = 
\left(\begin{array}{cccc} r\\ x_2 \end{array}\right)$. It follows by iterating the above procedure that we can get a matrix $\alpha$ which is a product of matrices of 
the form $\left(\begin{array}{cccc} 1& q\\ 0 & 1\end{array}\right)$, $\left(\begin{array}{cccc} 1& 0\\ q' & 1\end{array}\right)$ where $q, q' \in \mathbb{Z}$ such that 
$\alpha \left(\begin{array}{cccc} x_1\\ x_2 \end{array}\right) = \left(\begin{array}{cccc} 1\\ 0\end{array}\right)$. Then $\alpha^{-1} \left(\begin{array}{cccc} 1\\ 0 \end{array}\right) = \left(\begin{array}{cccc} x_1\\ x_2\end{array}\right)$ and $\alpha \in SL_2(\mathbb{Z})$.
\end{example}
This example motivates to define $E_n(A)$ (see 2.6.3).
\begin{definition}
 Let $A$ be a commutative ring with identity. Let $e_{ij}(\lambda), i\ne j$ be the $n \times n$ matrix in $SL_n(A)$ which has $1$  as its diagonal entries and $\lambda$
as its $(i, j)^{th}$ entry. Let $E_n(A)$ be the subgroup of $SL_n(A)$ generated by $e_{ij}(\lambda), i\ne j$. We call elements of $E_n(A)$ as elementary matrices.
 \end{definition}
For any Euclidean domain $A$, any unimodular row $(a_1, a_2, \cdots, a_n)\in A^{n}$ can be completed to an elementary matrix, where $n\geq 2$ \textit{i.e.} $(a_1, a_2, \cdots, a_n) \stackrel {E_n(A)}\sim (1,0, \cdots,0)$, where $\stackrel {E_n(A)}\sim$ denotes the induced action of $E_n(A)$ on unimodular rows. 

Note that $E_n(A)\subset SL_n(A) \subset GL_n(A)$, hence $(a_1, a_2, \cdots, a_n) \stackrel {E_n(A)}\sim (1,0, \cdots,0) \Rightarrow 
(a_1, a_2, \cdots, a_n) \stackrel {SL_n(A)}\sim (1,0, \cdots,0) \Rightarrow (a_1, a_2, \cdots, a_n) \stackrel {GL_n(A)}\sim (1,0, \cdots,0)$. Since $A = k[X]$ (where $k$ is a field) is a Euclidean domain, $(a_1, a_2, \cdots, a_n)\stackrel {E_n(A)}\sim (1,0, \cdots,0)$ for any unimodular row $(a_1, a_2$, $\cdots, a_n)$ $\in A^{n}$.
\begin{proposition}\label{2}
 Let $A$ be a ring and $(a_1, a_2$, $\cdots, a_n)$ $\in A^{n}$ be a unimodular row. Then $(a_1+\lambda X a_{2}, a_2,\cdots, a_n) \in {A[X]}^{n}$ is a unimodular row over $A[X]$, for every $\lambda \in A$.
\end{proposition}
\begin{proof}
 Since $(a_1, a_2, \cdots, a_n) \in A^{n}$ is a unimodular, there exists $(b_1, b_2, \cdots, b_n) \in A^{n}$ such that $\sum_{i = 1}^{n} a_{i}b_{i} = 1$. Take 
$c_1 = b_1, c_2 = b_2 - b_{1}\lambda X, c_3 = b_3, \cdots, c_n = b_n$, thus we have $c_{1}(a_1+\lambda Xa_{2}) + \sum_{i = 2}^{n} a_{i}c_{i} = 1$. Therefore 
$(a_1+\lambda X a_{2}, a_2,\cdots, a_n)$ is a unimodular row over $A[X]$.
\end{proof}
From Proposition \ref{2}, we can prove that for any matrix $\sigma = \prod_{i = 1}^{r} E_{ij}(\lambda) \\ \in E_n(A)$, $(a_1, a_2, \cdots, a_n)\sigma(X) \in {A[X]}^{n}$ is a unimodular over $A[X]$, where $\sigma(X) = \prod_{i = 1}^{r} E_{ij}(\lambda X)$.
\section{Topological fact}
\begin{lemma}\label{3}
 Let $T$ be a topological space. Then any map $f : T \longrightarrow T$ is homotopic to itself.
\end{lemma}
\begin{lemma}[\cite{JR}, Exercise 1, page 325]\label{4}
Let $T, T_1$ and $T_2$ be topological spaces. Suppose maps $h, h' : T \longrightarrow T_1$ and $k, k' : T_1 \longrightarrow T_2$ are homotopic. Then $k\circ h$ and $k'\circ h'$ are homotopic.
\end{lemma}
\begin{lemma}[\cite{JR}]\label{5}
 Let $T, T_1$ and $T_2$ be topological spaces. Suppose $k : T \longrightarrow T_1$ is a continuous map and  $F$ is a homotopy between maps $f, f' : T_2 \longrightarrow T$, then $k\circ F$ is a homotopy between the maps $k\circ f$ and $k\circ f'$.
\end{lemma}
\begin{definition}
Let $f : T \longrightarrow T'$ be a continuous map. We say that f is a homotopy equivalence if there exists a continuous map $g : T' \longrightarrow T$ such that $f\circ g$ is homotopic to the identity map $I_{T'}$ on $T'$ and  $g\circ f$ is homotopic to the identity map $I_T$ on $T$. Two spaces $T$ and $T'$ are said to be homotopically equivalent or of the same homotopy type if there exists a homotopy equivalence from one to the other.
\end{definition}
\begin{theorem}[\cite{SD}]\label{6}
If two spaces $T$ and $T'$ are homotopically equivalent and any continuous map from $T$ to $S^{n-1}$ is homotopic to a constant map, then  any continuous map from $T'$ to $S^{n-1}$ is homotopic to a constant map. 
\end{theorem}
\begin{theorem}[\cite{SD}]\label{7}
Let $T$ be a simplicial complex of dimension $\leq r$ and $T'$ be a closed subcomplex of $T$. Suppose $f : T' \longrightarrow S^r$ is a continuous map, Then $f$ can be extended to a continuous map $f' : T \longrightarrow S^r$ such that $f'|T' = f$
\end{theorem}
This topological fact can be proved by induction. One chooses a vertex of the complex $T$ that does not belong to $T'$ and chooses an arbitrary extension of $f$. Then one extends $f$ linearly to the edge of $T$ that does not belong to $T'$, then to the two faces etc.
\begin{theorem}[\cite{SD}]\label{8}
Let $T$ be a simplicial complex of dimension $n$ and $T'$ be a closed subcomplex of $T$. Then any continuous map $f : T' \longrightarrow \mathbb{R}^n - \{0\}$ can be extended to a continuous map $f' : T \longrightarrow \mathbb{R}^n$ such that $f'^{-1}(\{0\})$ is a finite set.
\end{theorem}
\begin{theorem}[\cite{HY}, Theorem 4.4, page 153]\label{9}
Let $T$ be a separable metric space and $T'$ be a closed subspace of $T$. Suppose two maps $f$ and $g$ from $T' \longrightarrow S^n$ are homotopic. If there exists an extension $f' : T \longrightarrow S^n$ of $f$, then there also exists an extension $g' : T \longrightarrow S^n$  of $g$ such that $f'$ and $g'$ are homotopic.
\end{theorem}
%
Let $I$ be an ideal of $k[X_1, X_2, \cdots, X_m]$. Then $V(I) = \{(x_1$, $x_2$, $\cdots,$ $x_m) \in k^{m} \mid f(x_1, x_2,\cdots,x_m) = 0$, for every $f\in I\}$. By Hilbert basis theorem, every ideal of $k[X_1, X_2,$ $\cdots$, $X_m]$ is finitely generated, so $V(I)$ is the set of common zeros of finitely many polynomials.


\begin{example}
Let $A = \mathbb{R}[X_1, X_2]$ and $I = (X_{1}^{2} + X_{2}^{2} - 1) \subset A$ be an ideal  , then $V(I) \cap \mathbb{R}^{2} = S^{1}$(real sphere).
\end{example}
 
 Since $V(I)\cap \mathbb{R}^m$ is the set of common zeros of finitely many polynomials, it is a closed set in the usual Euclidean topology in $\mathbb{R}^{m}$, where $I$ be an ideal of $\mathbb{R}[X_1, X_2, \cdots, X_m]$. 
More generally for any field $k$, there exists a topology on $k^{m}$, where the subsets of the form $V(I)$ are closed. This topology is called the Zariski topology on $k^{m}$. 

In topology Tietze extension theorem (\cite{JR}, Theorem 3.2, page 212) says that 
``Any continuous map of a closed subset of a normal topological space $T$ into the reals $\mathbb{R}$ may be extended to a continuous map of $T$ into $\mathbb{R}$". As an algebraic analogue ``any polynomial function on $V(I)\cap \mathbb{R}^m$ is the restriction of a polynomial function on $\mathbb{R}^m$".  
\begin{remark}\label{10}
 \begin{enumerate}
\item Let $A = \mathbb{R}[X_1, X_2, \cdots, X_m],~ T = Spec(A)$ and  $V(0) = \mathbb{R}^{m}$. Let $\vec{a} = (a_1, a_2, \cdots, a_n)\in A^{n}$. Then we have a continuous map $F_{\vec{a}} : \mathbb{R}^{m} \longrightarrow \mathbb{R}^{n}$ defined as 
$F_{\vec{a}} (x_1, x_2,\cdots, x_m) = (a_1(x_1, x_2, \cdots, x_m), \cdots$, $a_n(x_1, x_2, \cdots, x_m))$, for every $(x_1, x_2, \cdots, x_m) \in \mathbb{R}^{m}$.
  Similarly if $A = \mathbb{R}[X_1, X_2, \cdots, X_m]/I,~ T = Spec(A)$ and  $V_{I}(\mathbb{R}) = V(I) \cap \mathbb{R}^{m}$, where $I$ is an ideal of a real algebraic variety in $\mathbb{R}[X_1, X_2, \cdots, X_m]$. Then for $\vec{a} = (a_1, a_2, \cdots, a_n) \in A^{n}$, we have a continuous map $F_{\vec{a}} : V_{I}(\mathbb{R}) \longrightarrow \mathbb{R}^{n}$ defined as $F_{\vec{a}} (x_1, x_2,\cdots, x_m) = (a_1(x_1, x_2,\cdots, x_m)$, $\cdots, a_n(x_1, x_2,\cdots, x_m))$, for every $(x_1, x_2,\cdots, x_m) \in V_{I}(\mathbb{R})$.
The Hilbert Nullstellensatz says that if ``$A = k[X_1, X_2, \cdots, X_m]$, where $k$ is algebraically closed field and $a_1, a_2, \cdots, a_n \in A$, then  $a_1, a_2, \cdots$, $a_n$ have a common zero in $k^{m}$ if and only if the ideal $\langle a_1, a_2, \cdots, a_n \rangle \ne A$''.
Similar to the Hilbert Nullstellensatz if  $A = \mathbb{R}[X_1, X_2, \cdots, X_m] / I$ and  $\vec{a} = (a_1, a_2, \cdots, a_n)\in A^{n}$ is unimodular \textit{i.e.} $\sum_{i = 1}^{n} a_ib_i = 1$ for some $(b_1, b_2, \cdots, b_n)\in A^{n}$, then $a_1, a_2, \cdots, a_n$ do not simultaneously vanish at any point of  $V_{I}(\mathbb{R})$. So we have a map $F_{\vec{a}} : V_{I}(\mathbb{R}) \longrightarrow \mathbb{R}^{n}-\{0, \cdots, 0\}$.
\item
Now define a map 
$g :\mathbb{R}^{n} - \{(0, 0, \cdots, 0)\} \longrightarrow S^{n-1}$ as $g(x) = \dfrac{x}{||x||}$ for every $x \in \mathbb{R}^{n} - \{(0, 0, \cdots, 0)\}$, where
$||x||$ denotes norm of $x$. Thus we have a map $G_{\vec{a}} = g\circ F_{\vec{a}}: V_{I}(\mathbb{R}) \longrightarrow S^{n-1} (\subset \mathbb{R}^{n} - \{(0, 0, \cdots, 0)\})$. 
This shows that for any unimodular row $\vec{a} = (a_1, a_2, \cdots, a_n)$ over $A$, we have two continuous maps $F_{\vec{a}}$ and $G_{\vec{a}}$.
\\\noindent \textbf{Claim}: $F_{\vec{a}}$ and $G_{\vec{a}}$ are homotopic. 

Since identity map $Id$ on $\mathbb{R}^{n} - \{(0, 0, \cdots, 0)\}$ is homotopic to $g$ (straight line homotopy) and $F_{\vec{a}}$ is homotopic to itself, map
$Id \circ F_{\vec{a}} = F_{\vec{a}} : V_{I}(\mathbb{R}) \longrightarrow \mathbb{R}^{n} - \{(0, 0, \cdots, 0)\}$ is homotopic to $G_{\vec{a}} = g\circ F_{\vec{a}}: V_{I}(\mathbb{R})   \longrightarrow S^{n-1}$. This proves the claim.
\item Let $A[X] = (\mathbb{R}[X_1, X_2, \cdots, X_m] / I)[X]$, where $I$ is an ideal of a real algebraic variety in $\mathbb{R}[X_1, X_2, \cdots, X_m]$. Then any element $a \in I$ vanishes on $(x_1, x_2, \cdots, x_m, x)$ for any $x \in \mathbb{R}$ and $(x_1, x_2, \cdots, x_m) \in V_{I}(\mathbb{R})$. This shows that  the real points of the variety corresponding to $A[X]$ is $V_{I}(\mathbb{R})\times\mathbb{R}$, where $V_{I}(\mathbb{R})$ is the set of real points of the variety corresponding to $A$. Therefore any unimodular row $(a_1(X), a_2(X), \cdots, a_n(X)) \in {A[X]}^n$ gives two maps $F_{\vec{a}}[X] : V_{I}(\mathbb{R}) \times\mathbb{R} \longrightarrow \mathbb{R}^{n}-\{0, \cdots, 0\}$ and
$G_{\vec{a}}[X]: V_{I}(\mathbb{R}) \times\mathbb{R} \longrightarrow S^{n-1}$. 
\end{enumerate}
\end{remark}
Throughout this chapter for any unimodular row $\vec{a} = (a_1, a_2, \cdots, a_n),~ F_{\vec{a}}$ and $G_{\vec{a}}$ denote maps from $V_{I}(\mathbb{R}) \longrightarrow \mathbb{R}^{n} - \{(0, 0, \cdots, 0)\}$
and from $V_{I}(\mathbb{R})  \longrightarrow S^{n-1}$, respectively. Similarly for any unimodular row $\vec{a}[X] = (a_1(X), a_2(X), \cdots$, $a_n(X)), ~F_{\vec{a}}[X]$ and $G_{\vec{a}}[X]$ denote maps from $V_{I}(\mathbb{R}) \times \mathbb{R} \longrightarrow \mathbb{R}^{n} - \{(0, 0, \cdots, 0)\}$ and from $V_{I}(\mathbb{R}) \times\mathbb{R}   \longrightarrow S^{n-1}$, respectively.
\begin{proposition}\label{11}
 Let $\vec{a} = (a_1, a_2, \cdots, a_n)$ and $\vec{b} = (b_1, b_2, \cdots, b_n)$ be two unimodular rows over $A = \mathbb{R}[X_1, X_2, \cdots, X_m]/I$ such that $(a_1, a_2, \cdots, a_n) \stackrel {E_n(A)}\sim (b_1, b_2, \cdots, b_n)$, where $I$ is an ideal of a real algebraic variety in $\mathbb{R}[X_1, X_2, \cdots, X_m]$.
Then the corresponding mappings $F_{\vec{a}} : V_{I}(\mathbb{R}) \longrightarrow \mathbb{R}^{n} - \{(0, 0, \cdots, 0)\}$ and $F_{\vec{b}} : V_{I}(\mathbb{R}) \longrightarrow \mathbb{R}^{n} - \{(0, 0, \cdots, 0)\}$ are homotopic.
\end{proposition}
\begin{proof}
Since $\vec{a} \stackrel {E_n(A)}\sim \vec{b}$, there exists $\sigma = \prod_{i = 1}^{r} E_{ij}(\lambda) \in E_n(A)$ such that $\vec{a} \sigma = \vec{b}$.
Take $\sigma(X) = \prod_{i = 1}^{r} E_{ij}(\lambda X) E_n(A[X])$, so we have $\sigma(0) = I_n$ and $\sigma(1) = \sigma$. Thus $\vec{a}\sigma(0) = \vec{a}I_n = \vec{a}$ and $\vec{a} \sigma(1) = \vec{a} \sigma = \vec{b}$. From Proposition \ref{2}, $(a_1, a_2, \cdots, a_n)\sigma(X)$ is a unimodular row over $A[X]$, from Remark \ref{10} (3), we have map $F_{\vec{a}}[X] :V_{I}(\mathbb{R}) \times \mathbb{R}  \longrightarrow \mathbb{R}^{n} - \{(0, 0, \cdots, 0)\}$ such that $F_{\vec{a}}[0] = F_{\vec{a}}$ and $F_{\vec{a}}[1] = F_{\vec{b}}$. Thus the maps $F_{\vec{a}} : V_{I}(\mathbb{R}) \longrightarrow \mathbb{R}^{n} - \{(0, 0, \cdots, 0)\}$ and 
$F_{\vec{b}} : V_{I}(\mathbb{R}) \longrightarrow \mathbb{R}^{n} - \{(0, 0, \cdots, 0)\}$ are homotopic via $F_{\vec{a}}[X]$.
\end{proof}
\noindent\textbf{Note}: From Remark \ref{10} (2), maps $F_{\vec{a}}$ and $G_{\vec{a}}$ are homotopic and from Proposition \ref{11}, maps $F_{\vec{a}}$ and $F_{\vec{b}}$ are homotopic. Hence
$G_{\vec{a}}$ and $G_{\vec{b}}$ are also homotopic (\textit{i.e.} $G_{\vec{a}} \stackrel {homo}\sim F_{\vec{a}} \stackrel {homo}\sim F_{\vec{b}} \stackrel {homo}\sim G_{\vec{b}} \Longrightarrow G_{\vec{a}} \stackrel {homo}\sim G_{\vec{b}}$).

Now we will give underlying topological idea of some results on unimodular rows, which shows that how one can think of the following results from topological point of view.
\begin{lemma}\label{12}
 Let $\vec{a} = (a_1, a_2, \cdots, a_n) \in Um_n(A)$ with $a_1$ being a unit of $A$. Then $(a_1, a_2, \cdots, a_n) \stackrel {E_n(A)}\sim (1, 0, \cdots, 0)$.
\end{lemma}
\noindent \textbf{Underlying topological idea}: Suppose $A = \mathbb{R}[X_1, X_2, \cdots, X_m]/I$ and $V_{I}(\mathbb{R}) = V(I) \cap \mathbb{R}^{m}$.
 Since $(a_1, a_2, \cdots, a_n) \in A^n$ is a unimodular row over $A$, we have a map $G_{\vec{a}} : V_{I}(\mathbb{R})  \longrightarrow S^{n-1}$.
Since $a_1$ is a unit, $a_1b_1 = 1$ for some $b_1 \in A$ \textit{i.e.} $a_1$ does not vanish at any point of $V_{I}(\mathbb{R})$. Therefore for any element of $\vec{p} = (0, x_2, \cdots, x_n) \in S^{n-1}$,
there does not exist an element $\vec{x} \in V_{I}(\mathbb{R})$ such that $G_{\vec{a}}(\vec{x}) = \vec{p}$. In other words $G_{\vec{a}}$ is not surjective.
So Im$(G_{\vec{a}}) \subset S^{n-1} - \{\vec{p}\}$. Since $S^{n-1} - \{\vec{p}\}$ is contractible, map $G_{\vec{a}}$ is homotopic to a constant map.
\begin{lemma}\label{13}
 Let $\vec{a} = (a_1, a_2, \cdots, a_n) \in A^{n}$ be a unimodular row. Suppose $(a_1, a_2, \cdots$, $a_i)$, for any $i < n$, is unimodular. Then $(a_1, a_2, \cdots, a_n) \stackrel {E_n(A)}\sim (1, 0, \cdots, 0)$.
\end{lemma}
\noindent \textbf{Underlying topological idea}:
Suppose $A = \mathbb{R}[X_1, X_2, \cdots, X_m]/I$ and $V_{I}(\mathbb{R}) = V(I) \cap \mathbb{R}^{m}$. Since $(a_1, a_2, \cdots, a_n) \in A^n$ is a unimodular row over $A$, we have a map $G_{\vec{a}} : V_{I}(\mathbb{R})  \longrightarrow S^{n-1}$. Since $(a_1, a_2, \cdots, a_i)$ is a unimodular, 
$a_1, a_2, \cdots, a_i$ do not vanish simultaneously at any point of $V_{I}(\mathbb{R})$. Therefore for any element $\vec{q} \in S^{n-1}$, whose first $i$-th coordinates are zero,
there does not exist an element $\vec{x} \in V_{I}(\mathbb{R})$ such that $G_{\vec{a}}(\vec{x}) = \vec{q}$. Hence proof is similar to the proof of Lemma \ref{12}. 

The inclusion map from $S^1$ to $\mathbb{R}^{2} - \{0\}$ is not homotopic to a constant map. As an algebraic consequence we have the following:- 
\begin{example} \label{14}
The unimodular row $\vec{a} = (\textbf{x}_1, \textbf{x}_2) \in A^{2}$, where $A = \mathbb{R}[X_1, X_2]/(X_{1}^{2} + X_{2}^{2} - 1)$ satisfies the property that $(\textbf{x}_1, \textbf{x}_2)$ can not be transformed to $(1, 0)$ via an element of $E_2(A)$ \textit{i.e.} there does
not exist a matrix $\sigma \in E_2(A)$ such that  $(\textbf{x}_1, \textbf{x}_2)\sigma = (1, 0)$. 
\end{example}
\begin{proof} Assume contrary. Since $(\textbf{x}_1, \textbf{x}_2) \in Um_2(A)$ and $V_{(X_{1}^{2} + X_{2}^{2} - 1)}(\mathbb{R}) = V((X_{1}^{2} + X_{2}^{2} - 1)) \cap \mathbb{R}^{2} = S^1$, we have a map 
$F_{\vec{a}} : S^1  \longrightarrow \mathbb{R}^{2} - \{0\}$ which is a inclusion map.
 Suppose there exists $\sigma \in E_2(A)$ such that $(\textbf{x}_1, \textbf{x}_2)\sigma = (1, 0)$. Then it follows $(\textbf{x}_1, \textbf{x}_2)\sigma(X) \in Um_2(A[X])$ (\ref{2}). So we have a map 
$F_{\vec{a}}[X] : S^1 \times\mathbb{R} \longrightarrow \mathbb{R}^{2} - \{0\}$ such that $F_{\vec{a}}[0] = F_{\vec{a}}$ and $F_{\vec{a}}[1] =$ constant map. This shows that inclusion map $F_{\vec{a}}$ is homotopic to a constant map, which is not possible. Hence our assumption is not true.
\end{proof}
\begin{theorem}[\cite{HB}, Theorem 9.3]\label{15}
 Let $A$ be a Noetherian ring of dimension $d$. Let $(a_1, a_2, \cdots$, $a_n) \in A^{n}$ be a unimodular row with $n \geq d+2$. Then $(a_1, a_2, \cdots, a_n) \stackrel {E_n(A)}\sim (1, 0, \cdots, 0)$.
\end{theorem}
\noindent\textbf{Underlying topological idea}:
Let $A$ be the coordinate ring of a real algebraic variety of dimension $d$. Then the dimension of $V_{I}(\mathbb{R}) \leq d$. Since $(a_1, a_2, \cdots, a_n)$ is unimodular, we have a continuous map $G_{\vec{a}}: V_{I}(\mathbb{R})\longrightarrow S^{n-1}$. Assume that $V_{I}(\mathbb{R})$ is a simplicial complex. By simplicial approximation there exists a simplicial map $\psi : V_{I}(\mathbb{R})\longrightarrow S^{n-1}$ such that $\psi$ and $G_{\vec{a}}$ are homotopic via straight line homotopy. 
Since $n \geq d+2$, $n-1 \geq d+1$. Therefore $\psi$ is not surjective (because simplicial map can not raise dimension) \textit{i.e.} Im$(\psi) \subset S^{n-1} - \{\vec{p}\}$. 
Since $S^{n-1} - \{\vec{p}\}$ is contractible, map $\psi$ is homotopic to a constant map. Hence $G_{\vec{a}}$ is also homotopic to a constant map.  
\begin{theorem}\label{16}
  Let $A$ be a Noetherian ring of dimension $d$. Let $(a_1(X), a_2(X)$, $\cdots, a_n(X)) \in A[X]^{n}$ be a unimodular row with $n \geq d+2$. Then
 \begin{align*}
 (a_1(X), a_2(X), \cdots, a_n(X)) \stackrel {E_n(A[X])}\sim (1, 0, \cdots, 0). 
\end{align*} 
\end{theorem}
\noindent\textbf{Underlying topological idea}:
Let $A$ be the co-ordinate ring of a real algebraic variety of dimension $d$. Then the dimension of $V_{I}(\mathbb{R}) \leq d$. Assume that $V_{I}(\mathbb{R})$ is a simplicial complex. For $n \geq d+3$, theorem follows from Theorem \ref{15}. Since $(a_1(X), a_2(X), \cdots, a_n(X)) \in A[X]^{n}$ is unimodular, we have a continuous map $G_{\vec{a}}[X]: V_{I}(\mathbb{R})\times \mathbb{R} \longrightarrow S^{n-1}$. Consider inclusion map $i : V_{I}(\mathbb{R}) \longrightarrow V_{I}(\mathbb{R})\times \mathbb{R}$ and projection $p : V_{I}(\mathbb{R})\times \mathbb{R} \longrightarrow V_{I}(\mathbb{R})$. Then $p\circ i = Id_{V_{I}(\mathbb{R})}$ which is obviously homotopic to identity map on $V_{I}(\mathbb{R})$. On the other hand map $H(x, t) = t(i\circ p)(x) + (1-t)Id_{V_{I}(\mathbb{R})\times \mathbb{R}}$ gives a homotopy between $i\circ p$ and $Id_{V_{I}(\mathbb{R})\times \mathbb{R}}$.

 This shows that spaces $V_{I}(\mathbb{R})$ and $V_{I}(\mathbb{R})\times \mathbb{R}$ are homotopically equivalent. Also from Theorem \ref{15}, any map from $V_{I}(\mathbb{R}) \longrightarrow S^{n-1}$ is homotopic to a constant map. Therefore from Theorem \ref{6}, $G_{\vec{a}}[X]$ is also homotopic to a constant map.
 
 The following theorem is a particular case of the Lemma 4.2.13 (Chapter 4).
\begin{theorem}\label{17}
 Let $A$ be a Noetherian ring of dimension $\leq r$ and $J$ be an ideal of $A$. Let $(\overline{a_1}, \overline{a_2}, \cdots, \overline{a_{r+1}}) \in (A/J)^{r+1}$ be a unimodular row.
Then there exist $c_1, c_2, \cdots, c_{r+1}$ such that $\overline{c_i} = \overline{a_i}$ and $(c_1, c_2, \cdots, c_{r+1})$ is unimodular over $A$.
\end{theorem}
 
\noindent\textbf{Underlying topological idea}:
Suppose $A = \mathbb{R}[X_1, X_2, \cdots, X_m]/I$ and $V_{I}(\mathbb{R}) = V(I) \cap \mathbb{R}^{m}$. Assume that $V_{I}(\mathbb{R})$ is a simplicial complex. Take $V_{J}(\mathbb{R}) = V(J)\cap \mathbb{R}^{m}$. Then $V_{J}(\mathbb{R})$ is a closed subspace of $V_{I}(\mathbb{R})$. Since  $(\overline{a_1}, \overline{a_2}, \cdots, \overline{a_{r+1}}) \in (A/J)^{r+1}$ is a unimodular row, we have a continuous map $G_{\vec{a}} : V_{J}(\mathbb{R}) \longrightarrow S^r$. By Theorem \ref{7}, $G_{\vec{a}}$ can be extended to $G'_{\vec{a}} : V_{I}(\mathbb{R}) \longrightarrow S^r$ such that ${G'_{\vec{a}}}_{|V_{J}(\mathbb{R})} = G_{\vec{a}}$.   

Now we will give motivation about algebraic proof of Theorem \ref{17} by topological proof of Theorem \ref{7}.

Consider Spec($A$) as a simplicial complex with vertices as minimal prime ideals of $A$, edges as height one prime ideals of $A$ and triangles as height $2$ prime ideals etc. We say that an element $a \in A$ vanishes on an edge corresponding to $\mathfrak{p}$ if $a \in \mathfrak{p}$. Also assume $J = \langle a_{r+2} \rangle$.

Now $K = \{\mathfrak{p} \in Spec(A)~|~ a_{r+2} \in \mathfrak{p}\}$ is a sub-complex of Spec($A$) and we have a row $(a_1, a_2, \cdots, a_{r+1})$ which does not vanish on $K$.
 
 Now we choose vertices of Spec($A$) which do not belong to $K$ \textit{i.e.} choose minimal prime ideals $\mathfrak{p}_1, \mathfrak{p}_2, \cdots, \mathfrak{p}_s$ of $A$ such that $a_{r+2} \notin \mathfrak{p}_i, ~ 1 \leq i \leq s$. Thus $\langle a_1, a_{r+2} \rangle \notin \mathfrak{p}_i, ~ 1 \leq i \leq s$. Therefore there exists $\lambda_1 \in A$ such that  $a'_1 = a_1 + \lambda_1 a_{r+2} \notin \cup_{i=1}^{s}\mathfrak{p}_i$. Hence we have a row $(a'_1 , a_2, \cdots, a_{r+1})$ which does not vanish on $K'$, where $K' = K \cup (\cup_{i=1}^{s}\mathfrak{p}_i)$.
 
Now we choose edges  of Spec($A$) which do not belong to $K'$ \textit{i.e.} choose prime ideals $\mathfrak{q}_1, \mathfrak{q}_2, \cdots, \mathfrak{q}_l$ containing $a'_1$ but $a_{r+2} \notin \mathfrak{q}_i$. If no such prime ideal with the above property exists, then $(a'_1, a_2, \cdots, a_{r+1})$ does not vanish on Spec($A$). So we are done. Otherwise, since $a_{r+2} \notin \mathfrak{q}_i$, $\langle a_2, a_{r+2} \rangle \nsubseteq \mathfrak{q}_i,~ 1 \leq i \leq l \Rightarrow \langle a_2, a_{r+2} \rangle \nsubseteq \cup_{i=1}^{l}\mathfrak{q}_i$. So there exists $\lambda_2$ such that $a'_2 = a_2 + \lambda_2 a_{r+2} \notin \cup_{i=1}^{l}\mathfrak{q}_i$. Hence we have a row $(a'_1 , a'_2, \cdots, a_{r+1})$ which does not vanish on $K''$, where $K'' = K' \cup (\cup_{i=1}^{l}\mathfrak{q}_i)$. Continuing same procedure we get the results.

The following theorem is a particular case of the Lemma 4.2.5 (Chapter 4).
\begin{theorem}\label{18}
 Let $A$ be a Noetherian ring of dimension $n$ and $J$ be an ideal of $A$. Suppose $(\overline{a_1}, \overline{a_2}, \cdots, \overline{a_n}) \in (A/J)^n$ is a unimodular row.
 Then there exist $b_1, b_2, \cdots, b_n \in A$ such that $\overline{b_i} = \overline{a_i}$, for all $i$ and ideal $\langle b_1, b_ 2,\cdots, b_n \rangle$ has height $n$.
\end{theorem}
\noindent\textbf{Underlying topological idea}:
Suppose $A = \mathbb{R}[X_1, X_2, \cdots, X_m]/I$ and $V_{I}(\mathbb{R}) = V(I) \cap \mathbb{R}^{m}$. Assume that $V_{I}(\mathbb{R})$ is a simplicial complex. Take $V_{J}(\mathbb{R}) = V(J)\cap \mathbb{R}^{m}$. Then $V_{J}(\mathbb{R})$ is a closed subspace of $V_{J}(\mathbb{R})$. Since  $(\overline{a_1}, \overline{a_2}, \cdots, \overline{a_{r+1}}) \in (A/J)^{r+1}$ is a unimodular row, we have a continuous map $G_{\vec{a}} : V_{J}(\mathbb{R}) \longrightarrow S^r$. By Theorem \ref{8}, $G_{\vec{a}}$ can be extended to a continuous  map $G'_{\vec{a}} : V_{I}(\mathbb{R}) \longrightarrow S^r$ such that ${G'_{\vec{a}}}_{|V_{J}(\mathbb{R})} = G_{\vec{a}}$. 
\begin{lemma}\label{19}
 The canonical homomorphism of groups from $E_n(A)$ to $E_n(A/I)$ is surjective.
\end{lemma}
The lemma follows from the fact that generators $E_{ij}(\overline{\lambda})$ of $E_n(A/I)$ can be lifted to generators $E_{ij}(\lambda)$ of $E_n(A)$. 

We show by an example that the canonical homomorphism from $SL_n(A)$ to $SL_n(A/I)$ need not be surjective.
\begin{example}\label{20}
 Let $B = \mathbb{R}[X, Y], A = \mathbb{R}[X, Y] / (X^{2} + Y^{2} -1)$. Then the unimodular row $(\textbf{x}, \textbf{y}) \in A^{2}$ can not be lifted to a unimodular row over $B^{2}$.
\end{example}
\begin{proof}
Let $ \alpha = \left(\begin{array}{cccc} \textbf{x} & \textbf{y} \\ -\textbf{y} & \textbf{x}\end{array}\right) \in SL_2(A)$.

\noindent Claim: There does not exist $\beta \in SL_2(B)$ such that $\overline{\beta} = \alpha$.\\ Since $\alpha \in A$, we have a map $\phi : S^1 \longrightarrow  SL_2(\mathbb{R})$ defined by 
\begin{align*}
\phi(x_1, x_2) = \alpha(x_1, x_2) = \left(\begin{array}{cccc} \textbf{x}(x_1, x_2) & \textbf{y}(x_1, x_2)\\ -\textbf{y}(x_1, x_2) & \textbf{x}(x_1, x_2) \end{array}\right) = \left(\begin{array}{cccc} x_1 & x_2\\ -x_2 & x_1 \end{array}\right).
\end{align*} Let 
\begin{align*}
\beta = \left(\begin{array}{cccc} f_1(X, Y) & f_2(X, Y)\\ -f_3(X, Y) & f_4(X, Y) \end{array}\right)\in SL_2(B)
\end{align*} be a lift of $\alpha$.
Therefore we have a map $\Phi : \mathbb{R}^2 \longrightarrow SL_2(\mathbb{R})$ defined by
\begin{align*}
\Phi(x_1, x_2) = \beta(x_1, x_2) =  \left(\begin{array}{cccc} f_1(x_1, x_2) & f_2(x_1, x_2)\\ -f_3(x_1, x_2) & f_4(x_1, x_2) \end{array}\right)
\end{align*} which is clearly an extension of $\phi$. In particular considering the first row of $\alpha ~\&~ \beta$ we see that inclusion map from $S^1 \longrightarrow \mathbb{R}^2 - \{0\}$ extends to $\mathbb{R}^2 \longrightarrow \mathbb{R}^2 - \{0\}$ which is not possible. Hence claim is proved. Thus the unimodular row $(\textbf{x}, \textbf{y}) \in A^{2}$ can not be lifted to a unimodular row over $B^{2}$.\\
\textbf{Note}:  From Example \ref{20}, it is clear that $\left(\begin{array}{cccc} \textbf{x} & \textbf{y} \\ -\textbf{y} & \textbf{x}\end{array}\right)$ does not belong to $E_2(A)$ otherwise it could be lifted to a matrix in $E_2(B) \subset SL_2(B)$ \textit{i.e.} the unimodular row $(\textbf{x}, \textbf{y}) \in A^{2}$ is not elementary completable.
 \end{proof}
The general from of the Lemma \ref{19} is the following fact-
\begin{theorem}\label{21}
 Let $\vec{a} = (a_1, a_2, \cdots, a_n) \in A^{n}$ be a unimodular row. Suppose $J$ is an ideal of $A$ and $(\overline{a_1}, \overline{a_2}, \cdots, \overline{a_n}) \stackrel {E_n(\overline{A})}\sim 
(\overline{b_1}, \overline{b_2}, \cdots, \overline{b_n})$, where bar denotes reduction modulo $J$. Then there exists $(c_1, c_2, \cdots, c_n) \in A^{n}$ such that $(a_1, a_2, \cdots, a_n) \stackrel {E_n(A)}\sim (c_1, c_2, \cdots, c_n)$ and $(\overline{c_1}, \overline{c_2}, \cdots, \overline{c_n}) = (\overline{b_1}, \overline{b_2}, \cdots, \overline{b_n})$. In particular  $(\overline{b_1}, \overline{b_2}, \cdots, \overline{b_n})$ can be lifted to a unimodular row over $A$.
\end{theorem}
\noindent\textbf{Underlying topological idea}:
 Suppose $A = \mathbb{R}[X_1, X_2, \cdots, X_m]/I$ and $V_{I}(\mathbb{R}) = V(I) \cap \mathbb{R}^{m}$. Assume that $V_{I}(\mathbb{R})$ is a separable metric space. Take $V_{J}(\mathbb{R}) = V(J)\cap \mathbb{R}^{m}$. Then $V_{J}(\mathbb{R})$ is a closed subspace of $V_{I}(\mathbb{R})$.
Since $(\overline{a_1}, \overline{a_2}, \cdots, \overline{a_n}) \stackrel {\overline{\alpha} \in E_n(\overline{A})}\sim (\overline{b_1}, \overline{b_2}, \cdots, \overline{b_n})$, by Proposition \ref{11}, the corresponding continuous maps $G'_{\vec{a}} : V_{J}(\mathbb{R}) \longrightarrow S^{n-1}$ and $G'_{\vec{b}} : V_{J}(\mathbb{R}) \longrightarrow S^{n-1}$ are homotopic. That is there exists a continuous map $F : V_{J}(\mathbb{R}) \times \mathcal{I} \longrightarrow S^{n-1}$ such that $F(x, 0) = G'_{\vec{a}}$ and $F(x, 1) = G'_{\vec{b}}$.

Let $\alpha$ be the lift of $\overline{\alpha}$. Take $(c_1, c_2, \cdots, c_n) = (a_1, a_2, \cdots, a_n)\alpha$. Thus we have an extension $G_{\vec{a}} : V_{I}(\mathbb{R}) \longrightarrow S^{n-1}$ given by $(c_1, c_2, \cdots, c_n)$ of $G'_{\vec{a}}$. \\Define $H' : (V_{J}(\mathbb{R}) \times \mathcal{I})\cup(V_{I}(\mathbb{R}) \times \{0\}) \longrightarrow S^n$ by 
\begin{align*}
H'(x, t) = \left\{\begin{array}{rcl}F(x, t) & \mbox{for} & x \in V_{J}(\mathbb{R}) ~\&~ 0 \leq t \leq 1  \\ G_{\vec{a}}& \mbox{for} & x \in V_{I}(\mathbb{R}) ~\&~ t = 0 \end{array}\right.
\end{align*} which is obviously a continuous map. Since $(V_{J}(\mathbb{R}) \times \mathcal{I}) \cup (V_{I}(\mathbb{R}) \times \{0\})$ is a closed subset of $V_{I}(\mathbb{R}) \times \mathcal{I}$, $H'$ can be extended to $H : V_{I}(\mathbb{R}) \times \mathcal{I} \longrightarrow S^n$ (from proof of Theorem \ref{9}). Take $G_{\vec{b}} = H(x, 1) : V_{I}(\mathbb{R}) \longrightarrow S^n$. Then ${G_{\vec{b}}}_{|V_{J}(\mathbb{R})} = G'_{\vec{b}}$. 
\begin{example}
Let $A = \mathbb{R}[X, Y, Z]/(X^{2}+Y^{2}+Z^{2}-1)$. Then $(\textbf{x}, \textbf{y}, \textbf{z})\in A^{3}$ is a unimodular row. Since the identity map from $S^{2}$ to itself is not homotopic 
to a constant map, we have as an algebraic consequence that $(x, y, z)$ is not equivalent to $(1, 0, 0)$ via the action of $E_3(A)$. In other words $(\textbf{x}, \textbf{y}, \textbf{z})$ is not completable
to an elementary matrix. In fact $(\textbf{x}, \textbf{y}, \textbf{z})$ is not completable to a matrix in $GL_3(A)$.
\end{example}
\begin{proof}
 Assume contrary that $(\textbf{x}, \textbf{y}, \textbf{z})$ is the first row of a matrix in $GL_3(A)$ \textit{i.e.} $(\textbf{x}, \textbf{y}, \textbf{z})$ is a completable unimodular row. In other words $P \cong A^{3}/\langle \textbf{x}, \textbf{y}, \textbf{z} \rangle \cong  A^{2}$. Thus we have a surjective homomorphism $f : P \longrightarrow A$. Suppose $e_1, e_2, e_3$ are the standard basis vectors of $A^{3}$ and $f(\overline{e_i}) = h_i$ for $i = 1, 2, 3$. Since $f(\textbf{x}, \textbf{y}, \textbf{z}) = 0$,  we have $\textbf{x}f(\overline{e_1}) + \textbf{y}f(\overline{e_2}) + \textbf{z}f(\overline{e_3}) = 0 \Rightarrow \textbf{x}h_1 + \textbf{y}h_2 + \textbf{z}h_3 = 0 $. This implies that $Xh_1(X, Y, Z) + Yh_2(X, Y, Z) + Zh_3(X, Y, Z)$ is a multiple of $X^{2}+Y^{2}+Z^{2}-1$.

In particular, if $x_{1}^{2}+x_{2}^{2}+x_{3}^{2} = 1$, then $x_{1}h_1(x_1, x_2, x_3) + x_{2}h_2(x_1, x_2, x_3) + x_{3}h_3(x_1, x_2, x_3)$ = $0$ implies that $(h_1(x_1, x_2, x_3)$, $h_2(x_1, x_2, x_3)$, $h_3(x_1, x_2, x_3))$ is perpendicular to $(x_1, x_2, x_3)$. Define a continuous vector field $\Phi_1 : S^{2} \longrightarrow \mathbb{R}^{3}$ by $\Phi_1(x_1, x_2, x_3)= (h_1(x_1, x_2, x_3)$, $h_2(x_1, x_2, x_3)$, $h_3(x_1, x_2, x_3))$. The zeros of this vector field are those point $(x_1, x_2, x_3) \in S^2$ where $h_1(x_1, x_2, x_3)$, $h_2(x_1, x_2, x_3)$, $h_3(x_1, x_2, x_3)$ are all zero.
 Since $f(\overline{e_i}) = h_i$, the ideal $(h_1(\textbf{x}, \textbf{y}, \textbf{z}), h_2(\textbf{x}, \textbf{y}, \textbf{z})$, $h_3(\textbf{x}, \textbf{y}, \textbf{z})) = A$. Hence the corresponding vector field on $S^2$ has no real zeros, contradicting the fact that there is no nowhere vanishing continuous vector field on $S^2$.

Thus $(\textbf{x}, \textbf{y}, \textbf{z})$ is not completable to a
 matrix in $GL_3(A)$ implies that $(\textbf{x}, \textbf{y}, \textbf{z})$ is not completable to a matrix in $E_3(A)$.  
\end{proof}
\begin{proposition}[\cite{MKR}, Corollary 2.2]\label{221}
Let $(a, b, c) \in A^3$ be a unimodular row. Then $(a^2, b, c)$ is completable.
\end{proposition}
To understood underlying topological idea, we first give a proof of Proposition \ref{221}.
\begin{proof}
Since $(a, b, c)$ is a unimodular, there exist $a', b' ~\&~ c'$ such that $aa' + bb' + cc' = 1$. Consider the matrix
\begin{align*}
\alpha = \left(\begin{array}{cccc} 0 & a & b & c\\ -a & 0 & c' & -b'\\ -b & -c' & 0 & a' \\-c & b' & -a' & 0\end{array}\right).
\end{align*} Then $det(\alpha) = (aa' + bb' + cc')^2 = 1$.
Since $\langle 0, -a, -b, -c \rangle = A, ~ (0, -a, -b, -c)$ is a elementary completable. The matrix 
$\alpha_{1}\alpha_2$ is a completion of $(0, -a, -b, -c)$, where 
\begin{align*}
\alpha_1 = \left(\begin{array}{cccc} 1 & 0 & 0 & 0\\ a & 1 & 0 & 0\\ b & 0 & 1 & 0 \\ c & 0 & 0 & 1\end{array}\right) 
\end{align*}
and 
\begin{align*}
\alpha_2 = \left(\begin{array}{cccc} 1 & -a' & -b' & -c'\\ 0 & 1 & 0 & 0\\ 0 & 0 & 1 & 0 \\ 0 & 0 & 0 & 1\end{array}\right). 
\end{align*}
\vspace{-1cm}
Take $\alpha' = \alpha_{1}\alpha_{2}\alpha$. Then 
\begin{align*}
\alpha' = \left(\begin{array}{cccc} 1 & a & b & c\\ 0 & a^2 & ab + c' & ac - b'\\ 0 & ab-c' & b^2 & bc + a' \\ 0 & ac + b' & bc - a' & c^2\end{array}\right)
\end{align*}
 and $det(\alpha') = det(\alpha)$. By replacing $a \rightarrow a, b \rightarrow b, c \rightarrow c, a' \rightarrow a', b' \rightarrow b'+ac, c' \rightarrow  c'-ab$, we have
 \begin{align*}
 \beta = \left(\begin{array}{cccc} 1 & a & b & c\\ 0 & a^2 & c' & -b'\\ 0 & 2ab-c' & b^2 & bc+a' \\ 0 & 2ac+b' & bc-a' & c^2\end{array}\right).
 \end{align*}
 Then $det(\beta) = aa'+b(b'+ac)+c(c'-ab) = aa' + bb' + cc' = 1$. Again by replacing $a \rightarrow a, b \rightarrow -b', c \rightarrow c', a' \rightarrow a', b' \rightarrow -b, c' \rightarrow c$,
we have 
\begin{align*}
\sigma = \left(\begin{array}{cccc} 1 & a & -b' & c'\\ 0 & a^2 & c & b\\ 0 & -2ab'-c & b'^2 & -b'c'+a' \\ 0 & 2ac'-b & -b'c'-a' & c'^2\end{array}\right).
\end{align*}
Then $det(\sigma) = aa' + (-b')(-b+ac') + c'(c+ab') = aa' + bb' + cc' = 1$. It is clear that $det(\sigma) = det \left(\begin{array}{cccc}  a^2 & c & b\\ -2ab'-c & b'^2 & -b'c'+a' \\ 2ac'-b & -b'c'-a' & c'^2\end{array}\right) = 1$. Hence the matrix
\begin{align*}
\alpha_{3} = \left(\begin{array}{cccc}  a^2 & b & c\\ -2ac'-b & c'^2 & -b'c'+a' \\ 2ab'-c & -b'c'-a' & b'^2\end{array}\right) 
\end{align*} is a completion of $(a^2, b, c)$.
\end{proof}
\noindent\textbf{Underlying topological idea:}
Let $A = \mathbb{R}[X_1, X_2, \cdots , X_m]/I$ and $V_{I}(\mathbb{R}) = V(I) \cap \mathbb{R}^{m}$.
 Since $(a, b, c) \in A^3$ is a unimodular row, $(0, -a, -b, -c)$ is elementary completable. \\ 
Consider an exact sequence $0 \longrightarrow  SL_3(\mathbb{R}) \overset{i_1}\longrightarrow SL_4(\mathbb{R}) \overset{i_2} \longrightarrow \mathbb{R}^{4} - \{0\} \longrightarrow 0$, where $i_1$ sends any matrix $\sigma' \in SL_3(\mathbb{R})$ to $\left(\begin{array}{cccc}  1 & 0\\ 0 & \sigma' \end{array}\right) \in SL_4(\mathbb{R})$ and  $i_2$ sends any matrix $\sigma_1 \in SL_4(\mathbb{R})$ to its first row.

Since $\alpha \in SL_4(A)$, we have a map $\phi_1 : V_{I}(\mathbb{R}) \longrightarrow SL_4(\mathbb{R})$
defined as $\phi_1(\vec{x}) = \alpha(\vec{x})$ and for any unimodular row of length $4$, we have a map from $V_{I}(\mathbb{R}) \longrightarrow \mathbb{R}^{4} - \{0\}$. The map $i_2 \circ \phi_1 : V_{I}(\mathbb{R}) \longrightarrow \mathbb{R}^{4} - \{0\}$ is equal to the map given by unimodular row $(0, -a, -b, -c)$ which is homotopic to constant map (by Proposition \ref{11}). Also $\alpha_3 \in SL_3(A)$ gives a map $\phi_2 : V_{I}(\mathbb{R}) \longrightarrow SL_3(\mathbb{R})$ such that $i_1 \circ \phi_2$ homotopic to the map $\phi_1$ and the map $i'_2 \circ \phi_2 : V_{I}(\mathbb{R}) \longrightarrow \mathbb{R}^{3} - \{0\}$ is equal to the map given by unimodular row $(a^2, b, c)$, where $i'_2$ sends any matrix $\sigma_1 \in SL_3(\mathbb{R})$ to its first row. 
\section{On a lemma of Vaserstein's}
Throughout this section $T$ is a compact Hausdorff topological space (\textit{i.e.} normal space), $C(T)$ is the ring of real valued continuous functions on $T$ and $v_0 \stackrel {GL_n(C(T))}\sim v_t$ means there exists a matrix $\alpha \in GL_n(C(T))$ such that $v_0 \alpha = v_t$ , where $v_0$ and $v_t$ are unimodular row in $C(T)$. This is an equivalence relation.

We now give a proof of a result which says that if there is no nowhere vanishing continuous vector field on $S^2$, then $S^2$ is not contractible. This proof is motivated from Simha (\cite{RS}).

To begin proof we need some preliminaries on reflections:-

Let $w (\ne 0) \in \mathbb{R}^n$ be a vector. A reflection about $w$ is a linear transformation $\sigma : \mathbb{R}^n \longrightarrow \mathbb{R}^n$ which satisfies $\sigma(w) = -w, ~\&~ \sigma(w_1) = w_1$,
 where $w_1 \in w^{\bot} = \{ w'\in \mathbb{R}^n ~|~ \langle w', w \rangle = 0\}$. 

Let $v \in \mathbb{R}^n$ such that $v = v_1 + \lambda w, v_1 \in w^{\bot}$. Then $\sigma (v) = v_1 - \lambda w$. We have $\langle v, w \rangle = \langle v_1, w \rangle + \lambda \langle w, w \rangle = \lambda \langle w, w \rangle$, since $v_1 \in w^{\bot}$. This implies that 
$\lambda = \dfrac{\langle v, w \rangle}{\langle w, w \rangle}$. Therefore $\sigma(v) = v_1 -  \lambda w = v_1 + \lambda w - 2 \lambda w  = v - 2 \lambda w  = v - 2 \dfrac{\langle v, w \rangle}{\langle w, w \rangle} w$. This map $\sigma$ is denoted by $\sigma _w$. 

If $v_1$ and $w_1$ are two vectors in $\mathbb{R}^{n}$ and $||v_1|| = ||w_1||$, then we have a rhombus whose sides are $v_1, w_1$ and whose diagonal
are $v_1 + w_1, v_1 - w_1$. Thus $\sigma_{v_1 - w_1}(v_1)$ = $v_1 -2 \dfrac{\langle v_1, v_1 - w_1 \rangle}{\langle v_1 - w_1, v_1 - w_1 \rangle} (v_1 - w_1)$ = $v_1 - 
2 \dfrac{||v_1||^{2} - \langle v_1, w_1 \rangle}{||v_1||^{2} - 2\langle v_1, w_1 \rangle + ||w_1||^{2}} (v_1 - w_1)$  = $v_1 - (v_1 - w_1) = w_1$. Also we have $\sigma_{v_1 + w_1}(v_1) = -w_1$. 
\begin{remark}\label{22}
Any continuous map from $T \longrightarrow \mathbb{R}^n - \{0\}$ leads to a unimodular row $(a_1, a_2, \cdots$, $a_n)$ over the ring of continuous function $C(T)$, where $a_i$ is the projection from $T \longrightarrow \mathbb{R}$  because the element $\sum a_{i}^{2} \in  \langle a_1, a_2, \cdots, a_n \rangle$ is a continuous function on $T$ which does not vanish at any point of $T$ otherwise each $a_i$ will vanish at that point. In particular, any continuous map $T \longrightarrow S^{n-1}$ gives rise to a unimodular row over $C(T)$.  Let $H : T \times \mathcal{I} \longrightarrow S^{n-1}$ be a continuous map.\\
\noindent\textbf{Claim}: For sufficiently small $t$, $v_0 \stackrel {GL_n(C(T))}\sim v_t$, where $v_0$ and $v_t$ are the corresponding unimodular rows given by the maps $H(x, 0) : T \longrightarrow S^{n-1}$ and  $H(x, t) : T \longrightarrow S^{n-1}$, respectively. 
\begin{proof}
By continuity of $H$ and compactness of $T$, it follows that there exists a $t_\theta > 0$ such that for $t < t_\theta$, $v_0$ and $v_t$ are sufficiently close in the sense that $v_0(p)$ and $v_t(p)$ are not antipodal for every $p$ \textit{i.e.} $v_0(p) + v_t(p) \ne 0$ for $t < t_\theta$ and for every $p\in T$. Then for every $p \in T, \sigma_{v_0(p) + v_t(p)}$ is a reflection which is an element of $O_n(\mathbb{R})$ and satisfies $\sigma_{v_0(p) + v_t(P)}(v_0(p)) = -v_t(p)$. Hence there exists $\alpha : T \longrightarrow O_n(\mathbb{R})$ sending $p$ to $-\sigma_{v_0(p) + v_t(p)}$ such that $\alpha v_0 = v_t$. Since $O_n(\mathbb{R}) \subset GL_n(\mathbb{R}), \sigma \in GL_n(C(T))$ and $v_0 \stackrel {GL_n(C(T))}\sim v_t$.
\end{proof}
 \end{remark}
 Now we will give a proof of the fact that $S^2$ is not contractible.
\begin{proof} Assume contrary that the real two sphere $S^2$ is contractible. Then there exists a continuous map $H : S^{2} \times \mathcal{I} \longrightarrow S^2$, where $\mathcal{I} = [0, 1]$
such that the map $H(x, 0)$ is a constant map given by $(1, 0, 0)$ and $H(x, 1)$ is a identity map on $S^2$ corresponding to the unimodular row $(\textbf{x}, \textbf{y}, \textbf{z}) \in (C(S^2))^3$.

Now let $S = \{ t \in \mathcal{I} \mid v_0 \stackrel {GL_n(C(T))}\sim v_t\}$. By the compactness and connectedness of $\mathcal{I}$,  it is easy to see that $S = \mathcal{I}$ \textit{i.e.} $v_0 \stackrel {GL_n(C(T))}\sim v_1$. Hence from the claim it is clear that $(1, 0, 0)\stackrel {GL_3(C(S^2))}\sim (\textbf{x}, \textbf{y}, \textbf{z})$ \textit{i.e.} $(\textbf{x}, \textbf{y}, \textbf{z})$ is completable which is not possible. Hence $S^2$ is not contractible.
\end{proof}
In Section 5.2, we have seen  that if $A$ is the co-ordinate ring of a real algebraic variety and $\vec{a} \in A^n$ is unimodular, then $\vec{a}$ gives rise to a continuous function $F_{\vec{a}} : V_{I}(\mathbb{R}) \longrightarrow \mathbb{R}^{n} - \{0\}$. Also if $\vec{b} \in A^{n}$ is unimodular and $\vec{a} \stackrel {E_n(A)}\sim \vec{b}$, then the corresponding maps $F_{\vec{a}}, F_{\vec{b}} : V_{I}(\mathbb{R}) \longrightarrow \mathbb{R}^{n} - \{0\}$ are homotopic.

We shall now investigate the extent to which the converse is valid. Simha's proof (\cite{RS}) shows that 
if $T$ is a compact topological space and $G_{\vec{a}}, G_{\vec{b}} : V_{I}(\mathbb{R}) \longrightarrow S^{n-1}$ are continuous maps which are homotopic, then the corresponding unimodular rows $\vec{a} ~\&~ \vec{b}$ satisfy $\vec{a} \stackrel {GL_n(C(T))}\sim \vec{b}$.
To investigate the converse, we need a lemma of Vaserstein:-
\begin{lemma}[\cite{VS}]\label{23}
Let $A$ be a ring and $\vec{a} = (a_1, a_2, \cdots, a_n) \in A^{n}$ be a unimodular row. Let $\vec{b} = (b_1, b_2, \cdots, b_n)$ and $\vec{c} = (c_1, c_2, \cdots, c_n)$ be such that 
$\sum_{i = 1}^{n} a_ib_i = \sum_{i = 1}^{n} a_ic_i = 1$. Then there exists a matrix $\alpha \in SL_n(A)$ which can be connected to the identity matrix such that $\vec{b} \stackrel {\alpha}\sim \vec{c}$, if $n \geq 3$.
\end{lemma}
\noindent\textbf{Underlying topological idea}: If $A$ is the co-ordinate ring of a real algebraic variety and $\vec{a}, \vec{b}, \vec{c} \in A^n$ are unimodular rows satisfying the property $\vec{a}\vec{b}^t = \vec{a}\vec{c}^t = 1$. Then we have two continuous maps $F_{\vec{b}}, F_{\vec{c}} : V_{I}(\mathbb{R}) \longrightarrow \mathbb{R}^{n}-\{0\}$.\\
\noindent\textbf{Claim}: $F_{\vec{b}} ~\&~ F_{\vec{c}}$ are homotopic.\\
Since $\vec{a}\vec{b}^t = \vec{a}\vec{c}^t = 1$, for each point $p \in V_{I}(\mathbb{R})$, the vectors $F_{\vec{b}}(p)$ and $F_{\vec{c}}(p)$ are not in antipodal directions (if $F_{\vec{b}}(p) = - F_{\vec{c}}(p)$ means $\langle \vec{a}, \vec{b} \rangle = 1 = - \langle \vec{a}, \vec{c} \rangle$ which is not true). Therefore $H = (1-t)F_{\vec{b}} + tF_{\vec{c}}$ is a continuous map from $V_{I}(\mathbb{R})\times \mathcal{I}$ to $\mathbb{R}^{n}-\{0\}$, which is a homotopy between $F_{\vec{b}} ~\&~ F_{\vec{c}}$. Thus the claim is proved.

To give algebraic proof of Vaserstein's lemma one needs to compute the determinant of matrices of the kind $I_n + \alpha$, where $\alpha$ is a $n \times n$ matrix of rank $1$. Note that a general $n \times n$ matrix of rank $\leq 1$ over a field looks like $x^{t}y$, where $x = (x_1, x_2, \cdots, x_n)$ and $y = (y_1, y_2, \cdots, y_n)$. In particular a $2 \times 2$ matrix of rank $\leq 1$ over a field looks like $x^{t}y = \left(\begin{array}{cccc} x_1y_1 & x_1y_2 \\ x_2y_1 & x_2y_2\end{array}\right)$, where $x = (x_1, x_2)$ and $y = (y_1, y_2)$.
 
For simplicity we state and prove the lemma in the $2 \times 2$ case (the general case being similar).

 \begin{lemma}\label{24}
 \begin{align*}
det\left(\begin{array}{cccc} 1 + x_1y_1 & x_1y_2 \\ x_2y_1 & 1 + x_2y_2\end{array}\right) = 1 + x_1y_1 + x_2y_2. 
 \end{align*}  
\end{lemma} \textit{i.e.} $det(I_n + x^{t}y) = 1 + xy^{t}$
 \begin{proof}
 Note that
 \begin{align*}
 det\left(\begin{array}{cccc} 1 + x_1y_1 & x_1y_2 & 0\\ x_2y_1 & 1 + x_2y_2 & 0 \\ 0 & 0 & 1 \end{array}\right) = 
 det\left(\begin{array}{cccc} 1 + x_1y_1 & x_1y_2 & x_1\\ x_2y_1 & 1 + x_2y_2 & x_2 \\ 0 & 0 & 1 \end{array}\right) \\ = 
 det\left(\begin{array}{cccc} 1  & 0 & x_1\\ 0 & 1  & x_2 \\ -y_1 & -y_2 & 1 \end{array}\right) = 
 det\left(\begin{array}{cccc} 1 & 0 & x_1\\ 0 & 1 & x_2 \\ 0 & 0 & 1 + x_1y_1 + x_2y_2 \end{array}\right) = 1 + x_1y_1 + x_2y_2. 
\end{align*}
 Since 
 \begin{align*}
det\left(\begin{array}{cccc} 1 + x_1y_1 & x_1y_2 \\ x_2y_1 & 1 + x_2y_2\end{array}\right) = 
 det\left(\begin{array}{cccc} 1 + x_1y_1 & x_1y_2 & 0\\ x_2y_1 & 1 + x_2y_2 & 0 \\ 0 & 0 & 1 \end{array}\right).
 \end{align*} This completes proof of the lemma.
 \end{proof}
\begin{remark}\label{25}
 \begin{enumerate}
 \item The elementary matrices $E_{ij}(\lambda)$ are of the form $I_n + \alpha$, where $\alpha$ has rank $1$. 
 \item The matrix $E_{ij}(1) = I_n + e_{i}^t e_j, i \ne j$. If we conjugate $E_{ij}(1)$ by a matrix $\alpha \in GL_n(k), k$ a field, we obtain the matrix $\alpha E_{ij}(1)\alpha^{-1} = I_n + v^{t}w$, where $v^{t} = \alpha e_{i}^t, w = e_j\alpha^{-1}$. Since $e_ie_j^{t} = 0, vw^{t} = 0$ and by Lemma \ref{24}, it follows that the determinant of the matrix $I_n + v^{t}w$ is $1$.
\item Let $\sigma : k^{n} \longrightarrow k$ be a linear transformation and $w \in k^n$ is a non-zero vector such that $\sigma(w) \ne 0$. Define $\sigma' : k^{n} \longrightarrow k^{n}$ by $\sigma'(v) = v + \sigma(v)w$. Then $\sigma'$ is a linear transformation and its matrix is of the form $I_n + \alpha$, where $\alpha$ has rank $1$.
 Since dim $Ker(\sigma) = n-1$, we can choose a basis $x_1, x_2, \cdots, x_{n-1}$ of $Ker(\sigma)$ and therefore $x_1, x_2, \cdots, x_{n-1}, w$ is a basis of $k^n$. Thus the determinant of $\sigma'$ with respect to the above basis is seen to be $1 + \sigma(w)$.
In particular we can compute the determinant of the reflection transformation which sends $v$ to $v - 2 \dfrac{\langle v, w \rangle}{\langle w, w \rangle} w$. Note 
that in this case $\sigma(v) = - 2 \dfrac{\langle v, w \rangle}{\langle w, w \rangle}$, in particular $\sigma(w) = -2$. The determinant of the reflection transformation is $-1$.
\item If $\sigma(w) = 0$, then we have a basis $v_1, v_2, \cdots, v_{n-1}$ with $v_1 = w$ of $Ker(\sigma)$. Let $v_n \in k^n$ but $v_n \notin Ker(\sigma)$. The determinant
of the matrix corresponding to $\sigma'$ with respect to the basis $v_1, \cdots, v_{n-1}, v_n$ is $1$.
 In particular, it follows as we have seen before that the $det(I_n + v^{t}w) = 1$, where $wv^{t} = 0$.
 \end{enumerate}
\end{remark}
Now we prove Lemma \ref{23} of Vaserstein.
\begin{proof}
Suppose $\vec{a} \in A^{n}$ is unimodular and $\vec{b}, \vec{c} \in A^{n}$ are such that $\vec{a}\vec{b}^t = \vec{a}\vec{c}^t = 1$. Then $\vec{a}(\vec{c} - \vec{b})^{t} = 0$. So we have a matrix $\alpha = I_n + (\vec{c} - \vec{b})^{t}\vec{a}$ satisfying the property that $det(\alpha) = 1$ and $\alpha \vec{b}^{t} = \vec{b}^{t} + (\vec{c} - \vec{b})^t = \vec{c}^t$.
 Take $\beta =  I_n + (\vec{c} - \vec{b})^{t}\vec{a}X \in SL_n(A[X])$. Then $\beta(0) = I_n, \beta(1) = \alpha$ proving the lemma.
\end{proof}
\begin{theorem}
Let $H : T \times \mathcal{I} \longrightarrow S^{n-1}$ be a continuous homotopy such that $v_0$ and $v_1$ are the corresponding unimodular rows given by maps $H(x, 0) : T \longrightarrow S^{n-1}$ and  $H(x, 1) : T \longrightarrow S^{n-1}$, respectively. Then there exists a matrix $\alpha \in SL_n(C(T))$ such that $\alpha$ can be connected to the identity and $v_0 \stackrel {\alpha}\sim v_1$. 
\end{theorem} 
\begin{proof}
 Since $\mathcal{I}$ is a compact and connected space, it suffices to prove that if $H(x, t) : T \longrightarrow S^{n-1}$ is a continuous map for sufficiently small $t$, then $v_{0} \stackrel {\alpha}\sim v_t$, where $\alpha \in SL_n(C(T))$ and $\alpha$ can be connected to the identity matrix.\\
We choose small $t$ so that $v_{0}(p) + v_{t}(p) \ne 0$ for every $p \in T$. Define $W : T \longrightarrow S^{n-1}$ by $W(p) = \dfrac{v_{0}(p) + v_{t}(p)}{1 + v_{0}(p)v_{t}(p)}$. Since $v_{0}(p)$ and $v_{t}(p)$ are not antipodal, $v_{0}(p)v_{t}(p) \ne -1$. Since any continuous map $T \longrightarrow S^{n-1}$ gives rise to a unimodular row over $C(T)$, we have $v_{0}(p)W(p) = \dfrac{v_{0}(p)v_{0}(p) + v_{0}(p) v_{t}(p)}{1 + v_{0}(p)v_{t}(p)} = \dfrac{1 + v_{0}(p) v_{t}(p)}{1 + v_{0}(p)v_{t}(p)} = 1$. Similarly $v_{t}(p)$ $.W(p) = 1$. By Vaserstein's Lemma, there exists $\alpha \in SL_n(C(T))$ such that $\alpha$ can be connected to the identity matrix and $v_0 \stackrel {\alpha}\sim v_t$.
\end{proof}
\noindent \textbf{Note}:Consider quaternion algebra\\ $Q = \{x_1 + ix_2 + jx_3 + kx_4 ~|~ x_1, x_2, x_3, x_4 \in \mathbb{R}\}$ over $\mathbb{R}$. Let
$q_1 = x_1 + ix_2 + jx_3 + kx_4$. Define $\varphi : Q \longrightarrow Q$ by $T(q) = q_1q$. Then $\varphi$ is a linear transformation. Also $\varphi(1) = x_1 + ix_2 + jx_3 + kx_4,~ \varphi(i) = ix_1 - x_2 - kx_3 + jx_4,~ \varphi(j) = jx_1 + kx_2 - x_3 - ix_4$ and $\varphi(k) = kx_1 - jx_2 + ix_3 - x_4$. Therefore 
\begin{align*}
matrix (\varphi) = \left(\begin{array}{cccc} x_1 & -x_2 & -x_3 & -x_4\\ x_2 & x_1 & -x_4 & x_3\\ x_3 & x_4 & x_1 & -x_2 \\x_4 & -x_3 & x_2 & x_1\end{array}\right).
\end{align*}  In particular, if $q_1 = ix_2 + jx_3 + kx_4$, then 
\begin{align*}
matrix(\varphi) = \left(\begin{array}{cccc} 0 & -x_2 & -x_3 & -x_4\\ x_2 & 0 & -x_4 & x_3\\ x_3 & x_4 & 0 & -x_2 \\x_4 & -x_3 & x_2 & 0\end{array}\right),
\end{align*} which is a skew symmetric matrix. Hence we have a map $\phi' : \mathbb{R}^{3} \longrightarrow S'$, where $S'$ is the set of all $4 \times 4$ skew symmetric matrices over $\mathbb{R}$ defined by
\begin{align*}
\phi'(x, y , z) = \left(\begin{array}{cccc} 0 & -x & -y & -z\\ x & 0 & -z & y\\ y & z & 0 & -x \\z & -y & x & 0\end{array}\right).
\end{align*}  Also we have a map $\phi = \phi'_{| S^2} : S^2 \longrightarrow S$,
where $S$ is the set of all $4 \times 4$ skew symmetric matrices of determinant $1$ defined 
by 
\begin{align*}
\phi(x, y , z) = \left(\begin{array}{cccc} 0 & -x & -y & -z\\ x & 0 & -z & y\\ y & z & 0 & -x \\z & -y & x & 0\end{array}\right).
\end{align*}

Consider $A = \mathbb{R}[X_1, X_2, \cdots, X_n]/I$, where $I$ is an ideal of real algebraic variety in $\mathbb{R}[X_1, X_2$, $\cdots$, $X_n]$.           
 Suppose $(a_1, a_2, a_3)$ and $(b_1, b_2, b_3)$ are two unimodular rows over $A$ such that $(a_1, a_2, a_3) \stackrel {E_3(A)}\sim (b_1, b_2, b_3)$. Then the corresponding maps $G_{\vec{a}} : V(\mathbb{R}) \longrightarrow S^{2}$ and $G_{\vec{b}} : V(\mathbb{R}) \longrightarrow S^{2}$ are homotopic. Therefore maps $V(\mathbb{R}) \overset{G_{\vec{a}}} \longrightarrow  S^{2} \overset{\phi} \longrightarrow Q$ and $V(\mathbb{R}) \overset{G_{\vec{b}}} \longrightarrow  S^{2} \overset{\phi} \longrightarrow Q$ are also homotopic.

Let $A$ be a commutative ring with identity. Suppose $\vec{a} = (a_1, a_2, a_3)$ is a unimodular row, then there exists $\vec{b} = (b_1, b_2, b_3) \in A^3$ such that $\langle \vec{a}, \vec{b} \rangle = \vec{a}\vec{b}^{t} = \sum_{i = 1}^{3} a_ib_i = 1$. Therefore we have a skew symmetric matrix 
\begin{align*}
V(\vec{a}, \vec{b}) = \left(\begin{array}{cccc} 0 & a_1 & a_2 & a_3\\ -a_1 & 0 & b_3 & -b_2\\ -a_2 & -b_3 & 0 & b_1 \\-a_3 & b_2 & -b_1 & 0\end{array}\right)
\end{align*}.

Let $S =$ Set of all $4 \times 4$ skew symmetric matrices over $A$. Define a relation $\sim$  on $S$ as $\alpha_1 \sim \alpha_2$ if $\alpha_2 = \beta^{t}\alpha_1\beta$ for some $\beta \in E_4(A)$.
Clearly this is an equivalence relation on $S$. On the other hand  $\stackrel {E_3(A)}\sim$ is an equivalence relation on $Um_3(A)$. 

We define a relation $\Psi : Um_3(A)/E_3(A) \longrightarrow S/\sim$ by $\Psi(\vec{a}) = V(\vec{a}, \vec{b})$.

\textbf{Claim}: $\Psi$ is a well defined map.

Let $\vec{a} = (a_1, a_2, a_3)$ and $\vec{a'} = (a'_1, a'_2, a'_3)$ be such that $\vec{a'} = \vec{a}\tau$, for some $\tau \in E_3(A)$

Suppose $\vec{a}\vec{b}^t = \sum_{i = 1}^{3} a_ib_i = 1$. Then 
\begin{align*}
V(\vec{a}, \vec{b}) = \left(\begin{array}{cccc} 0 & a_1 & a_2 & a_3\\ -a_1 & 0 & b_3 & -b_2\\ -a_2 & -b_3 & 0 & b_1 \\-a_3 & b_2 & -b_1 & 0\end{array}\right).
\end{align*}
Since $\vec{a}\vec{b}^t = 1, ~\vec{a}\tau{\tau}^{-1}\vec{b}^t = 1$. This implies that $\vec{a'}\vec{b'}^t = 1$, where $\vec{b'} = \vec{b}({\tau}^{-1})^t$. We know that $\tau = \prod_{i = 1}^{r} E_{ij}(\lambda)$, where $E_{ij}(\lambda), i\ne j$ is the $3 \times 3$ matrix in $SL_3(A)$ which has $1$ as its diagonal entries and $\lambda$
as its $(i, j)^{th}$ entry. So it suffices to prove the claim when 
\begin{align*}
\tau = \left(\begin{array}{cccc} 1 & \lambda & 0\\ 0 & 1 & 0 \\ 0 & 0 & 1 \end{array}\right).
\end{align*} Then 
\begin{align*}
\tau^{-1} = \left(\begin{array}{cccc} 1 & -\lambda & 0\\ 0 & 1 & 0 \\ 0 & 0 & 1 \end{array}\right).
\end{align*} Therefore 
\begin{align*}
V(\vec{a'}, \vec{b'}) = \left(\begin{array}{cccc} 0 & a_1 & a_1\lambda + a_2 & a_3\\ -a_1 & 0 & b_3 & -b_2\\ -a_1\lambda-a_2 & -b_3 & 0 & b_1-\lambda b_2 \\-a_3 & b_2 & -b_1 + \lambda b_2 & 0\end{array}\right).
\end{align*} Let 
\begin{align*}
\beta = \left(\begin{array}{cccc} 1 & 0\\ 0 & \tau^{t}\end{array}\right).
\end{align*} Then 
\begin{align*}
\beta^{t}V(\vec{a}, \vec{b}) = \left(\begin{array}{cccc} 0 & a_1 & a_2 & a_3\\ -a_1 & 0 & b_3 & -b_2\\ -a_1\lambda-a_2 & -b_3 & b_3\lambda & a_1-\lambda b_2 \\-a_3 & b_2 & -b_1 & 0\end{array}\right).
\end{align*} Hence 
\begin{align*}
\beta^{t}V(\vec{a}, \vec{b})\beta =  \left(\begin{array}{cccc} 0 & a_1 & a_1\lambda + a_2 & a_3\\ -a_1 & 0 & b_3 & -b_2\\ -a_1\lambda-a_2 & -b_3 & 0 & b_1-\lambda b_2 \\-a_3 & b_2 & -b_1 + \lambda b_2 & 0\end{array}\right) = V(\vec{a'}, \vec{b'}).
\end{align*} Thus $\Psi$ is well defined.

\end{document}